\documentclass[11pt]{article}
\usepackage{color, graphicx}
\usepackage{amsmath}
\usepackage{amssymb}

\usepackage{charter}
\usepackage[T1]{fontenc}
\usepackage{textcomp}
\usepackage{calrsfs}

\newcommand{\runninghead}[1]{\gdef\RH{#1}}\newcommand{\RH}{}
\newcommand{\msc}[1]{\gdef\MSC{#1}}\newcommand{\MSC}{}
\renewcommand{\title}[1]{\gdef\TT{#1}}\newcommand{\TT}{}
\renewcommand{\author}[1]{\gdef\AU{#1}}\newcommand{\AU}{}
\newcommand{\address}[1]{\gdef\ADR{#1}}\newcommand{\ADR}{}
\renewcommand{\date}[1]{\gdef\DD{#1}}\newcommand{\DD}{}
\newcommand{\version}[1]{\gdef\VER{#1}}\newcommand{\VER}{}

\renewcommand{\maketitle}{
\par \noindent {\footnotesize \upshape Running head: \RH \hfill \DD \par
\noindent Math.\ Subj.\ Class. (2000): \MSC \hfill\textbf{\VER}}
\vspace{3cm}
\begin{center} {\LARGE \normalfont \TT} \par \medskip \AU \end{center} \vspace{1cm}
}

\newcommand{\finalinfo}{
\bigskip \noindent {\small \upshape \ADR} }

\makeatletter
\def\@seccntformat#1{\csname the#1\endcsname. \ }
\makeatother

\usepackage{amsthm}
\theoremstyle{plain}
\newtheorem{introprop}{Proposition}
\newtheorem{introthm}[introprop]{Theorem}

\swapnumbers
\theoremstyle{plain}
\newtheorem{thm}{Theorem}[section]
\newtheorem{prop}[thm]{Proposition}
\newtheorem{cor}[thm]{Corollary}
\newtheorem{lem}[thm]{Lemma}
\theoremstyle{definition}

\newtheorem{rem}[thm]{Remark}
\newtheorem{notation}[thm]{Notation}
\newtheorem{notationassumptions}[thm]{Notation and Assumptions}

\newtheorem{se}[thm]{}

\usepackage[british]{babel}

\newcommand{\bad}{{\operatorname{bad}}}
\newcommand{\tors}{{\operatorname{tors}}}
\newcommand{\calA}{{\mathcal A}}
\newcommand{\calB}{{\mathcal B}}

\newcommand{\calO}{{\mathcal O}}
\newcommand{\calP}{{\mathcal P}}

\newcommand{\calS}{{\mathcal S}}
\newcommand{\calT}{{\mathcal T}}
\newcommand{\calU}{{\mathcal U}}
\newcommand{\calV}{{\mathcal V}}
\newcommand{\calW}{{\mathcal W}}

\newcommand{\F}{{\mathbb F}}

\newcommand{\Q}{{\mathbb Q}}
\newcommand{\R}{{\mathbb R}}
\newcommand{\Z}{{\mathbb Z}}

\newcommand{\pp}{{\mathfrak p}}
\newcommand{\qq}{{\mathfrak q}}

\newcommand{\nn}{{\mathfrak n}}
\newcommand{\dd}{{\mathfrak d}}
\newcommand{\be}{\begin{enumerate}}
\newcommand{\ee}{\end{enumerate}}
\newcommand{\norm}{{\mathbf N}}

\def\ord{\mathop{\mathrm{ord}}\nolimits}
\def\Fr{\mathcal{F}}
  \newcommand{\rank}{\mbox{rank}}
\usepackage[OT2,OT1]{fontenc}
\newcommand\cyr{%
\renewcommand\rmdefault{wncyr}%
\renewcommand\sfdefault{wncyss}%
\renewcommand\encodingdefault{OT2}%
\normalfont
\selectfont}
\DeclareTextFontCommand{\textcyr}{\cyr}

\begin{document}

\runninghead{Defining the integers}
\version{version 2.0}
\date{January 23, 2007}
\msc{03B25, 11U05}
\title{Defining the integers in large rings of a number  \\[1mm] field using one universal
quantifier}
\author{\textit{by} Gunther Cornelissen \textit{and} Alexandra Shlapentokh} % ||OFF||
 
\address{(gc) Mathematisch Instituut, Universiteit Utrecht, Postbus
80.010, 3508 TA Utrecht, Nederland, \\ email: {\tt
cornelis@math.uu.nl}

\medskip

\noindent (as) Department of Mathematics,
East Carolina University,
Greenville, NC 27858-4353, U.S.A. \\ email: {\tt shlapentokha@ecu.edu}}
% break with \medskip in case of more authors

% \author{Gunther Cornelissen \par \ADR} % ||ON||
\maketitle

\bigskip

{\footnotesize \hfill \emph{dedicated to Yuri Matiyasevich on his 60th birthday}

\hfill \dots the undecidable poem ``{{\cyr V Peterburge  my soi0demsya snova}}" \dots (\cite{Wesling})}

% Vp2 Peterburgd1  my soi0demsya snova

\bigskip

\begin{abstract}
\noindent Julia Robinson has given a first-order definition of the rational integers $\Z$ in the rational numbers
$\Q$ by a formula $(\forall \exists \forall \exists)(F=0)$ where the $\forall$-quantifiers run over a total of 8
variables, and where $F$ is a polynomial.

We show that for a large class of number fields,  not including $\Q$, for  every
$\varepsilon>0$, there exists a set of
primes $\calS$ of natural density exceeding $1-\varepsilon$, such that $\Z$ can be defined as a subset of the ``large'' subring  %
$$\{x \in K \, : \, \ord_{\pp}x \geq 0,\ \forall \, \pp \not \in \calS \}$$%
 of $K$  by a formula  where there is only one
$\forall$-quantifier.     In the case of $\Q$ we will need two quantifiers.  We also show that in
some cases one can define a subfield of a number field using just one universal quantifier.
\end{abstract}%

\section*{Introduction} % || OFF ||
% \section{Introduction % || ON ||
\paragraph{Julia Robinson's work} In 1949 (\cite{Robinson:49}) Julia Robinson  showed that the set of integers $\Z$
is definable in the language of rings in the field of rational numbers $\Q$ by a first-order formula. This implies that the
first order theory of fields in undecidable.

The quantifier complexity of this formula was analysed in \cite{CZ}: it is equivalent to a formula of the form
$$ (\forall x^{(1)}_{1}  \dots x^{(1)}_{5})( \exists y^{(1)}_{1} \dots y^{(1)}_{4})(\forall x^{(2)}_{1}
   \dots x^{(2)}_{3})( \exists y^{(2)}_{1}) \ : \ F({\bf x}, {\bf y}) = 0, $$
where $F$ is a polynomial over $\Z$ in multi-variables $${\bf x}= (x_{1}^{(1)},\dots,x^{(2)}_{5}) \mbox{ and }
 {\bf y}= (y_{1}^{(1)},\dots,y_4^{(1)}, y_{1}^{(2)}).$$
Given the results of Julia Robinson one may ask by just how complex a formula can $\Z$ be defined in $\Q$? Here,
``complex'' refers to how many quantifiers of what kind need to be used, and how many quantifier alterations are necessary.
In view of Hilbert's Tenth Problem, it is particularly relevant to reduce the number of universal quantifiers --- since a
definition of $\Z$ in $\Q$ without any universal quantifiers would imply that Hilbert's Tenth Problem for $\Q$ has a
negative answer, and disprove a conjecture of Mazur on the topology of rational points, cf.\ \cite{Mazur}. The above formula
has a total of eight universal quantifiers and three quantifier alterations.

\paragraph{Recent developments} (a) In \cite{CZ}, it was shown that a (heuristically probable) conjecture about elliptic curves allows one
 to give a \emph{model} of $\Z$ over $\Q$ involving only one universal quantifier.
 Here, a model is essentially a countable definable set over $\Q$ with a bijection to $\Z$ such that via this bijection,
 the graphs of addition and multiplication on $\Z$ are definable subsets over $\Q$.

(b) Recall that any ring in between $\Z$ and $\Q$ is of the form $\Z[\frac{1}{S}]$ for some set of primes $S$.
In \cite{Poonen}, Poonen showed that there is a set of primes $S$ of full natural density such that $\Z$ has a model in $\Z[\frac{1}{S}]$ involving no
universal quantifiers whatsoever. Note however that this does not settle the question of defining $\Z$ as a diophantine \emph{subset} of $\Z[\frac{1}{S}]$.
The result was extended to number fields in \cite{PS}.

(c) Let $K$ be a number field and let $\calW_K$ be a set of primes of $K$. Define
$O_{K,\calW_K}$ to be the following ring:
\[%
O_{K,\calW_K} := \{x \in K: \ord_{\pp}x \geq 0, \,\, \forall \pp \not \in \calW_K\}.
\]%
If $\calW_K$ is infinite we will call these rings ``big'' or ``large''. The second named author has
given an existential definition of $\Z$ in some large subrings of the following fields: totally
real fields, their extensions of degree 2 and fields such that there exists an elliptic curve
defined over $\Q$, of positive rank over $\Q$ and of the same rank over the field in question (see
\cite{ Sh36, Sh33, Sh1, Sh6, Sh3}). The \emph{density of the set of inverted primes} in these
subrings \emph{is}, however, always \emph{bounded away from 1}, essentially by
$1-\frac{1}{[K:\Q]}$.

(d) In \cite{Po4} Poonen showed that $\Z$ can be defined in any number field using \emph{two} universal quantifiers and there are
exist  big subrings of $\Q$ with the set of inverted primes of density arbitrarily close to one and where $\Z$ is definable using
just \emph{one} universal quantifier.

\paragraph{Main results}
The aim of this work is to consider the problem of improving some of the above results in the
following sense: (a) unconditional on any conjectures;  (b) for subrings $\Z[\frac{1}{S}]$, where
$S$ is ``large'' in the sense of arbitrary high density $<1$; (c) such that it defines the actual
\emph{subset} $\Z$ of this large subring; (d) for other number fields instead of $\Q$.

The main results are as  follows:

\begin{introthm}
Let $K \not = \Q$ be a number field of one of the following types:%
\be%
\item $K$ is totally real;%
\item $K$ is an extension of degree two of a totally real number field;%
\item There exists an elliptic curve defined over $\Q$ and of positive rank over $\Q$ such that this curve
preserves its rank over $K$;

\ee%
Then for every $\varepsilon>0$, there exists a set of primes $\calW_K$ of $K$ of natural density exceeding
$1-\varepsilon$, such that $\Z$ can be defined as a subset of $O_{K,\calW_K}$  by a formula with
 only one $\forall$-quantifier.%
\end{introthm}

\begin{introthm}
 Let $K$ be a number field, including $\Q$.  Assume there exists an elliptic curve defined over $K$
of rank 1 over $K$. Then for every $\varepsilon>0$, there exists a set of primes $\calW_K$ of $K$ of natural density exceeding
$1-\varepsilon$, such that $\Z$ can be defined as a subset of $O_{K,\calW_K}$  by a formula with
 only two $\forall$-quantifiers.%
  \end{introthm}
 
Observe that the fact that $\Z$ can be defined over $\Q$ using two quantifiers does
not imply directly that $\Z$ can be defined over a ring of integers using two quantifiers: in translating a definition
over $\Q$ to a ring of integers, one has to represent a rational number as a ratio of two elements of the ring.
Thus a ``mechanical'' translation of Poonen's result over $\Q$ would produce a definition with four
universal quantifiers.

As we have mentioned above, in \cite{Po4}, Poonen also proved a that integers can be defined using just one quantifier over a big
subring of $\Q$. His result was obtained by different techniques and the sets of inverted primes are different from ours: in
\cite{Po4}, the inverted primes are inert in a finite union of quadratic extensions, whereas in the above theorem, primes without
relative degree one factors in a fixed extension are inverted, together with a density zero set related to the elliptic curve used
in the construction. These results raise the question of characterization of large subrings of $\Q$ in which $\Z$ admits a
diophantine definition, or a diophantine model, or a definition or model using $n \geq 1$ universal quantifiers; and in
particular whether there is any difference between these rings. Results like those above and in \cite{Po4} should be seen as a
first contribution to this type of questions.

We also prove the following theorems concerning definability with only one quantifier.

\begin{introthm}%
Let $K\not = \Q$ be a number field.  Assume there exists an elliptic curve defined over $K$
of rank 1 over $K$. Then for every $\varepsilon>0$, there exists a set of primes $\calW_K$ of $K$ of natural density exceeding
$1-\varepsilon$, such that $\Q \cap O_{K,\calW_K}$ can be defined over $O_{K,\calW_K}$  by a formula with
 only one  $\forall$-quantifier.%
\end{introthm}%

\begin{introthm}
\label{down}
Let $M/K$ be a number field extension.  Assume there exists an elliptic curve $E$
defined over $K$  such that $\rank E >0$ and $[E(M):E(K)] < \infty$.  Let $\calW_M$ be any set
of $M$ primes (including the set of all $M$-primes  and the empty set).  Then $O_{M,\calW_M} \cap
K$ is definable over  $O_{M,\calW_M}$ using just one universal quantifier.
\end{introthm}

\section{Elliptic Curves and Existential Models of $(\Z, +, |)$ over $O_{K,\calW}$}
\begin{se}
In this section we will use elliptic curves to define divisibility in large rings.  Most of
the technical details are taken from \cite{Po} and \cite{PS}.
 \end{se}
\begin{notation}
\label{S:notation section}
The following notation will be used for the rest of this section.
\begin{itemize}%
\item $K$ is a number field.%
\item $E$ is an elliptic curve of rank~1 defined over $K$ (in particular, we assume such an
$E$ exists).%
\item We fix a Weierstrass equation $W: y^2=x^3+ax+b$ for $E$ with coefficients in the ring of integers of $K$. %
\item $E(K)_\tors$ is the torsion subgroup of $E(K)$.%
\item $t$ is an even multiple of $\#E(K)_\tors$.%
\item $Q \in E(K)$ is such that $Q$ generates $E(K)/E(K)_\tors$.%
\item $P:=tQ$.%
\item $\calP_{\Q}=\{2,3,5,\dots\}$ is the set of rational primes.%
\item $\calP_K$ is the set of all finite primes of $K$.%
\item Let ${\mathcal S}_\bad = \calS_{\bad}(W,P,K) \subseteq \calP_K$ consist of the primes that ramify in $K/\Q$,
the primes for which the reduction of the chosen Weierstrass model is singular (this includes all primes above
$2$), and the primes at which the coordinates of $P$ are not integral.%
\item $h_K$ is the class number of $K$.
\item Write $nP=(x_n,y_n)=(x_n(P), y_n(P))$ where $x_n,y_n \in K$.%
\item Let the divisor of $x_n(P)$ be of the form%
\[%
\frac{{\mathfrak a}_n}{{\mathfrak d}_n}{\mathfrak b}_n=\frac{{\mathfrak a}_n(P)}{{\mathfrak
d}_n(P)}{\mathfrak b}_n(P)
\]%
where
\begin{itemize}%
\item ${\mathfrak d}_n=\prod_{{\mathfrak q}}{\mathfrak q}^{-a_{{\mathfrak q}}}$, where the product
is taken over all primes ${\mathfrak q}$ of $K$ not in $\calS_{\bad}$ such that $a_{{\mathfrak
q}}=\ord_{{\mathfrak q}}x_n <0$.%
\item ${\mathfrak a}_n=\prod_{{\mathfrak q}}{\mathfrak q}^{a_{{\mathfrak q}}}$, where the product is
taken over all primes ${\mathfrak q}$ of $K$ not in $\calS_{\bad}$ such that $a_{{\mathfrak
q}}=\ord_{{\mathfrak q}}x_n >0$.%
\item ${\mathfrak b}_n = \prod_{{\mathfrak q}}{\mathfrak q}^{a_{{\mathfrak q}}}$, where the product
is taken over all primes ${\mathfrak q}\in \calS_{\bad}$ and $a_{{\mathfrak q}}=\ord_{{\mathfrak
q}}x_n$.%
\end{itemize}%
\item For $n$ as above, let ${\calS}_n = {\calS}_n(P)=\{\pp \in {\mathcal P}_K : \pp | {\mathfrak
d}_n\}$. By
definition of $\calS_\bad$ and ${\mathfrak d}_n$, we have $\calS_1=\emptyset$.%
\item For $\ell \in \calP_{\Q}$, define $a_\ell$ to be the smallest positive number such that $\ell^{a_{\ell}} >
C$, where $C$ is defined in Proposition \ref{prop:biggerS} below. For all but finitely many primes $\ell$ we
have that $a_\ell =1$. %
\item For $j \in \Z_{\geq 1}$, $\ell \in \calP_{\Q}$,  let $\pp_{\ell^{j}}(P)=\pp_{\ell^j}$ be a prime of largest norm in
$\calS_{\ell^{j}} \setminus \calS_{\ell^{j-1}}$, if such a prime exists.%
\item Let $m_0 = \prod_{a_l >1}\ell^{a_{\ell}}$.
\item Let $\calV_K  =\calV_K(P) = \{\pp_{\ell^j}: \ell \in \calP_{\Q}, j \in \Z_{>0}\}$.%
\item Let $\calW_K \subset \calP_K$ satisfy the following requirements:  $\calV_K \subseteq \calP_K
\setminus \calW_K$ and $\calS_{\bad} \subset \calW_K$.
\item For $n$ as above, let $d_n = \norm_{K/\Q} {\mathfrak d}_n \in \Z_{\ge 1}$.%
\end{itemize}%
\end{notation}%
The following results can be found in \cite{PS}.%
  \begin{lem}
\label{le:orderchange}
Let $n \in \Z_{\ge 1}$.
Suppose that ${\mathfrak t} \in \calP_K$ divides ${\mathfrak d}_n$,
and $p$ is a rational prime.
\be
 \item If ${\mathfrak t} \mid p$,
then $\ord_{{\mathfrak t}}{\mathfrak d}_{pn}=\ord_{{\mathfrak t}}{\mathfrak d}_n+2$.
 \item If ${\mathfrak t} \nmid p$, then $\ord_{{\mathfrak t}}{\mathfrak d}_{pn}= \ord_{{\mathfrak t}}{\mathfrak d}_n$.
\ee
Consequently if $j \Big{|} k$ then $\dd_j \Big{|} \dd_k$. \qed
\end{lem}
In \cite{PS} it is assumed that $p \not =2$ but the proof is unchanged in that case also.
\begin{prop}[divisibility properties]%
\label{le:Po3.1}%
Let ${\mathfrak R}$ be an integral divisor of $K$. Then $$\{n \in \Z \setminus \{0\}: {\mathfrak R} \mid {\mathfrak
d}_n(P)\} \cup \{0\}$$ is a subgroup of $\Z$. \qed
\end{prop}%
\begin{prop}[growth rate]
 \label{le:denomheight}%
There exists $a \in \R_{>0}$ such that $\log d_n =(a-o(1))n^2$ as $n \longrightarrow \infty$.\qed
\end{prop}
\begin{prop}[existence of primitive divisors]
\label{prop:biggerS}
There exists $C >0$ such that for all $\ell, m \in \calP_{\Q}$ with $\max(\ell,m) >C$ we have that
$\calS_{\ell m} \setminus (\calS_\ell \cup \calS_m) \ne \emptyset$.\qed
\end{prop}

The following corollaries are easy consequences of the propositions above.
\begin{cor}[strong divisibility]%
\label{cor:intersec}%/
Let $m,n \in \Z \setminus \{0\}$, and let $(m,n)$ be their GCD. Then ${\mathcal S}_m \cap {\mathcal S}_n =
{\mathcal S}_{(m,n)}$. In particular, if $(m,n)=1$ then ${\mathcal S}_m \cap {\mathcal S}_n=\emptyset$.\qed
\end{cor}%

\begin{cor}%
\label{cor:div}
For any $z< k \in \Z_{>0}$ the following statements are true:
\be
\item  $\calS_{km_0} \setminus \calS_{zm_0} \not = \emptyset$.
\item  $\calS_{zm_0} \subset \calS_{km_0}$ if and only if $z$ divides $k$.
\item $\displaystyle \pp_{\ell^{j+\ord_{\ell}m_0}}$ exists for all $j >0$.
\item $\displaystyle \pp_{\ell^{j+\ord_{\ell}m_0}} \in \calS_{km_0}$ if and only if $\ell^j$ divides
$k$.\qed %
\ee
\end{cor}%

\begin{proof}%
\begin{enumerate}%
\item Since $\calS_{km_0} \cap \calS_{zm_0}=S_{(k,z)m_0}$ by Corollary \ref{cor:intersec}, without loss of generality we can
assume that $z \Big{|} k$.  By construction, $m_0 \geq C$, where $C$ is the constant from Proposition
\ref{prop:biggerS}. Thus, this part of the corollary holds. %
\item This assertion follows directly from Corollary \ref{cor:intersec}. %
\item To insure existence of $\displaystyle \pp_{\ell^{j+\ord_{\ell}m_0}}$, we need to show that
\[%
\calS_{\ell^{j+\ord_{\ell}m_0}}\setminus \calS_{\ell^{j-1+\ord_{\ell}m_0}} \not = \emptyset.
\]%
By construction either $\ell >C$ or $\ell^{\ord_{\ell}m_0} >C$. Thus this assertion follows from Proposition \ref{prop:biggerS} also.%
\item By Corollary \ref{cor:intersec} we have that $\displaystyle \pp_{\ell^{j+\ord_{\ell}m_0}} \in \calS_{km_0}$ only if
$\ell^{j+\ord_{\ell}m_0} \Big{|} km_0$ if and only if $j \Big{|} k$. Conversely, if $j \Big{|} k$, then $\ell^{j+\ord_{\ell}m_0}
\Big{|} km_0$ and $\calS_{\ell^{j+\ord_{\ell}m_0}} \subseteq \calS_{km_0}$ by Part (2) of the corollary. Thus we also have
\[%
\pp_{\ell^{j+\ord_{\ell}m_0}} \in \calS_{\ell^{j+\ord_{\ell}m_0}} \subseteq \calS_{km_0}.
\]%
\end{enumerate}%

\end{proof}%

\begin{cor}
\label{cor:less}
For all $n \in \Z_{>0}$, for some positive constant $\kappa$, independent of $n$, we have that $n^2 <  \kappa d_n$.\qed
\end{cor}

We now proceed to define divisibility in a large ring.
\begin{lem}
\label{le:defdiv}
The equations
\begin{equation}
\label{eq:1}
x^{h_K}_{km_0} = \frac{ a_k}{ b_k}; \ \ A_ka_k +B_kb_k = 1,
\end{equation}

\begin{equation}
\label{eq:2}
x^{h_K}_{jm_0} = \frac{ a_j}{ b_j}; \ \ A_ja_j +B_jb_j = 1,
\end{equation}
\begin{equation}
\label{eq:divide}
b_k = b_jz
\end{equation}
have a solution $ A_j, B_j, A_k, B_k, a_j, b_j, a_k, b_k \in
O_{K,\calW_K}$ if and only if $j$ divides $k$ in $\Z$.
\end{lem}%
\begin{proof}%
Observe that the set
\[%
\{(a,b) \in O_{K,\calW_K}\, :\, (\exists n \in \Z_{>0})( x_{m_0n} = \frac{a}{b})\}
\]%
is certainly diophantine over $O_{K,\calW_K}$ given that we know how to define the set of non-zero elements of the
ring.

Now suppose first that the equations are satisfied in $O_{K,\calW_K}$. Let
\[%
j=\prod\ell_i^{n_i}, n_i >0.
\]%
Then by Corollary \ref{cor:div} we have that $\pp_{\ell_i^{n_i+\ord_{\ell_i}m_0}}$ exists. Further, since $(a_j,b_j)=1$ in
$O_{K,\calW_K}$, and since for each $i$ we have that
\[%
\pp_{\ell_i^{n_i+\ord_{\ell_i}m_0}} \not \in \calW_K
\]%
 and
\[%
\ord_{\pp_{\ell_i^{n_i+\ord_{\ell_i}m_0}}}x_{jm_0}<0
\]%
it is the case that
\[%
\ord_{p_{\ell_i^{n_i+\ord_{\ell_i}m_0}}}b_j >0
\]%
 and
\[%
\ord_{\pp_{\ell_i^{n_i+\ord_{\ell_i}m_0}}}b_k >0.
\]%
Since $(a_k,b_k)=1$ in $O_{K,\calW_k}$, we conclude that
\[%
\ord_{\pp_{\ell_i^{n_i+\ord_{\ell_i}m_0}}}x_{km_0} <0
\]%
or
\[%
\pp_{\ell_i^{n_i+\ord_{\ell_i}m_0}} \in \calS_{km_0}.
\]%
But by Corollary \ref{cor:div} this is possible only if $\ell_i^{n_i}$ divides $k$.  Thus, if the
equations hold we have that $j$ divides $k$.

Conversely, suppose $j$ divides $k$. By the definition of the class number, we can let $a_k, b_k$ and $a_j,b_j$ be
pairs of algebraic integers relatively prime in $O_K$. Observe that ${\mathfrak d}^{h_K}_{jm_0}$ and ${\mathfrak
d}^{h_K}_{km_0}$ are precisely the non-invertible--in--$O_{K,\calW_K}$--parts of the  divisors of $b_j$ and $b_k$
respectively. Thus by Corollary \ref{cor:div} and Lemma \ref{le:orderchange} we have that $b_j$
divides $b_k$.
\end{proof}%

Summarizing the results of this section we state the following theorem:
\begin{thm}
\label{divmodel}
$(\Z, +, |)$ has a Diophantine model over $O_{K,\calW_K}$.
\end{thm}

We finish this section with a ``vertical'' definability result which exploits our ability to define
divisibility existentially.  First we need several technical propositions.

\begin{prop}%
\label{extensions}
Let $M/K$ be a number field extension of degree $n$.  Let $\mathfrak Q$ be a prime of $K$ and let
$\qq_1, \ldots, \qq_m$ be all the  primes of $M$ lying above $\mathfrak Q$.  Let $\alpha \in M$ be a
generator of $M$ over $K$ such that  $\alpha$ is  integral with respect to $\mathfrak Q$.  Let $u
\in M$ be integral at $\mathfrak Q$.  Assume further  there exists a sequence $\{k_i, y_i\}$ where
$k_{i+1} >k_i$ and $y_i \in K$ and $\ord_{\qq_j}y_i \geq 0$.  Finally assume that for all $i,j$ we
have that $\ord_{\qq_j}(u-y_i) \geq k_i$.    Then $u \in K$.
\end{prop}  %
\begin{proof}%
Let $D$ be the discriminant of the power basis of $\alpha$.
Using this power basis we can write
\[%
u=\sum_{r=0}^{n-1}a_r\alpha^r
\]%
with $Da_r \in K$ and integral at $\mathfrak Q$. Then
\[%
u-y_i=(a_0-y_i) +\sum_{r=1}^{n-1}a_r
\]%
and
\[%
\ord_{\qq_j}(u-y_i) >k_i, j=1,\ldots, m.
\]%
This implies that $\ord_{\mathfrak Q}a_r > \frac{k_i}{n}-\ord_{\mathfrak Q}D$ for all $i \in
\Z_{>0}$. Thus, $a_r=0, r=1,\ldots,n-1$ and $u \in K$.
\end{proof}%

The following two proposition are taken from \cite{Po}.

\begin{lem}%
\label{le:equiv}
There exists a positive integer $m_1$ such that for any positive integers $k,l $,
\[%
\dd(x_{lm_1}) \Big{|} \nn\left (\frac{x_{lm_1}}{x_{klm_1}}-k^2\right)^2
\]%
in the integral divisor semigroup of $K$. \qed
\end{lem}%

\begin{lem}%
\label{multiple}
Let $\mathfrak J$ be an integral divisor of $K$.  Then for some $m$ we have that
$\mathfrak J$ divides $\dd(x_m)$ in the integral divisor semigroup of $K$. \qed
\end{lem}%
We are now ready to prove the definability result.
 \begin{thm}%
 \label{rankonedown}
  Let $K$ be a number field.  Then  $O_{K,\calW_K}\cap \Q$ is definable over
$O_{K,\calW_K}$ using one universal quantifier. \qed
\end{thm}%
\begin{proof}%
If $K=\Q$, then the statement of the theorem is trivial.  So without loss of generality we can
assume that $K \not = \Q$.  Let $q$ be a rational prime and let $\qq_1,\ldots,\qq_r$ be all the
factors of $q$  in $K$. Let $u \in K$ be such that $\forall v \in K \exists x_{jm_1m_0}, x_{\ell
m_1m_0}$ such that \begin{enumerate} \item $j \Big{|} \ell$, \item for all $i=1,\ldots r$ we have
that  $\ord_{\qq_i}x_{km_1m_0} < \ord_{\qq_i}v$, \item for all $i=1,\ldots r$ we have that
\begin{equation}%
\label{eq:need}%
-\ord_{\qq_i}x_{jm_1m_0} <2 \ord_{\qq_i}\left (\frac{x_{j m_1 m_0}}{x_{\ell m_1m_0}}-u\right).
\end{equation}%
\end{enumerate}%
 We claim that $u \in \Q$.
 
 Indeed, fix $v$, consider the corresponding $x_{jm_1m_0}, x_{\ell m_1 m_0}$  and let $k
=\frac{\ell}{j}$. Then by Lemma \ref{le:equiv}
\[%
-\ord_{\qq_i}v <- \ord_{\qq_i}x_{jm_1m_0} <2 \ord_{\qq_i}\left (k^2-u \right).
\]%
Keeping in mind that $v$ is arbitrary, we can apply Proposition \ref{extensions} to reach the desired conclusion.

Suppose now that $u$ is a square of a rational integer $k$. Then using Lemma \ref{multiple} and Lemma \ref{le:orderchange},
for any $v$ we can find $j >0$ such that
\[%
\ord_{\qq_i}x_{jm_0m_1} < \ord_{\qq_i}v
\]%
for all $i=1,\ldots,m$. Set $\ell = kj$ and Lemma \ref{le:equiv} assures us that (\ref{eq:need}) will hold. Finally we remind the
reader that every positive integer can be written as a sum of squares and a every rational number
is a ratio of integers.

\end{proof}%

\section{From Divisibility to Multiplication}

In this section we will address the issue of converting our existential  model of $(\Z, +,|)$ to a
model of $(\Z, +, \times)$.  We will use the same notation as above.  Our starting point is the
following lemma
\begin{lem}
[\cite{CZ}, section 4]%
\label{le:complexity}%
There exists a formula $\Fr(l,m,n)$ in $(\Z,+,|,\neq)$ of the form $(\exists \forall \exists){\mathcal G}$ with one universal quantifier and $\mathcal G$
a  formula which is a conjunction of divisibility conditions and additions, such that for integers
$m,n$, we have $l=m \cdot n \iff \Fr(l,m,n)$. \hfill $\Box$
\end{lem}%

In our model of addition and divisibility we send a non-zero integer to triples $\{(x,y,z) \in
O_{K,\calW_K}^3\}$ where $(\frac{x}{z},\frac{y}{z})$ are affine coordinates of points on $E(K)$
with respect to our fixed affine Weierstrass equation $W$. Thus, a direct translation of ``$\forall
n \in \Z$'' becomes ``for all $(x,y,z)$ such that $(\frac{x}{z},\frac{y}{z})$ satisfy $W$'', which
uses three universal quantifiers. However, a result of Poonen from \cite{PO5} can be used to reduce
the number of $\forall$-quantifiers by one. Indeed, in \cite{PO5} it is shown that the set of
non-squares of a number field is Diophantine. Thus, we get
\begin{lem}%
A sentence of the form \textup{``for all $(x,y,z) \in O_{K,\calW_K}$ such that $(\frac{x}{z},
\frac{y}{z})$ satisfy $W$''} is equivalent to a sentence of the form
\[%
\forall x,z \in O_{K,\calW_K}:  ((z \not =0) \land(\exists y \in O_{K,\calW_K}: \frac{x^3}{z^3}+
\frac{ax}{z}+ b = \frac{y^2}{z^2})    \lor
\]%
 \[%
   ((x^3z^3+ axz^6+ bz^5 \mbox{\textup{ is not a square in }} K) \lor z=0),
\]%
involving only two universal quantifiers.
\end{lem}

\begin{rem}
The proof of Theorems 4.2 and 5.3 in \cite{CZ} contains a gap that can be fixed by the same technique: expressions of the form
``forall $(x,y) \in \Q^2$ satisfying $W$'' can be replaced by
\[%
\forall x \in \Q ((\exists y \in \Q: W(x,y)=0)   \lor (x^3+ ax+ b \mbox{\textup{ is not a square in }} \Q)).
\]%
\end{rem}

Combining the discussion above with Theorem \ref{le:complexity} we obtain the following.
\begin{thm}%
\label{thm:mult}
The set
\[%
\Pi=\{(A,B,C,D, Y, F) \in O_{K,\calW_K}| \exists j, k, z \in \Z_{>0}: \left\{ \begin{array}{c} z=jk
\\ x_{jm_0}=A/D \\ x_{km_0}=B/Y \\ x_{zm_0} = C/F \end{array} \right\} \}
\]%
is definable in $\calO_{K,\calW_K}$ by a $\exists \forall \exists$--formula using two universal
quantifiers. \qed
\end{thm}%

This theorem says that there is a model of the integers over the big ring involving two universal
quantifiers. We now have to work a bit more to define the actual subset of integers by a similar
formula.

\section{From Models to Subset-definitions}
In this section we will use Theorem \ref{thm:mult} to define the actual set of integers in large rings, using
two universal quantifier. First we extend our list of notation and assumptions.
\begin{notationassumptions}%
The notation and assumptions  below will be used in the remainder of the paper.
\begin{itemize}%
\item  $n=[K:\Q]$.%

\item  $L$ is any extension of $K$ of degree $r >0$.%
\item $\gamma \in O_L$ generates $L$ over $K$.
\item $d$ is an integer greater than $|\gamma -\sigma(\gamma)|$ for any embedding $\sigma$ of $L$ into its algebraic closure. %
\item  $G(T)=G_0(T)$ is the monic irreducible polynomial of $\gamma$ over $K$.
\item  $G_i(T)=G_0(T-di), i=1,\ldots,n$.%
\item Assume $\calW_K$ contains only the primes of $K$ without relative degree one factors in $L$
and not dividing the discriminant of $G_0$ (and consequently of any $G_i, i=1,\ldots,n$.)%
\item Assume $\calS_{\bad} \subset \calW_K$.%
\item Let $l_0=0,\ldots,l_{rn}$ be distinct natural numbers. %
\item For $x \in K$ let $\nn(x)=\prod\limits_{\pp \in \calP_K, \ord_{\pp} x >0}\pp^{\ord_{\pp} x}$
and let $\dd(x)=\prod\limits_{\pp \in \calP_K, \ord_{\pp} x <0}\pp^{-\ord_{\pp} x}$.     Similarly,
if $\mathfrak e$ is a divisor of $K$, we will denote the numerator of ${\mathfrak e}$ by
$\nn(\mathfrak e)$ and the denominator of ${\mathfrak e}$  by $\dd(\mathfrak e)$.%
\item Let ${\mathfrak U}, \mathfrak G$ be integral divisors of $K$ such that $\mathfrak
G^2=\mathfrak U$. Then we will denote $\mathfrak G$ by $\sqrt{\mathfrak U}$.
\end{itemize}%
\end{notationassumptions}%

We now go over some technical facts. The proof of the following four lemmas can be found in
\cite{Po} or \cite{Sh33}.

\begin{lem}%
\label{le:relprime} With $m_1$ as in Lemma \ref{le:equiv},
$(\dd(x_{lm_1}), \nn(x_{klm_1}))=1$ in the integral divisor semigroup of $K$. \qed
\end{lem}%

From Lemma \ref{le:equiv} and Lemma \ref{le:relprime} we also deduce the following corollary.
\begin{cor}%
\label{cor:h_K}%
\[%
\dd(x_{lm_1}) \Big{|} \nn\left (\frac{x^{h_K}_{lm_1}}{x^{h_K}_{klm_1}}-k^{2h_K}\right)^2
\]%
\end{cor}%
\begin{lem}%
\label{le:square}
For any $k \in \Z_{>0}$ we have that $\dd(x_k)$, $\dd_k$ are squares of some integral divisors of $K$. \qed
\end{lem}
Finally we state a lemma which will give us a handle on the bounds.  The proof can be found in
Chapter 5 of \cite{Sh34}.
\begin{lem}%
\label{le:bounds} Let $v \in O_{K,\calW_K}$, let $\alpha, \beta \in O_K$ be relatively prime in $O_K$, $v$ not a
unit of $O_{K,\calW_K}$. Assume
\[%
\frac{v^{h_K}}{\prod_{i=0}^{rn}G_i(\alpha/\beta-l_i)} \in O_{K,\calW_K}.
\]%
Then  $v^{h_K} = yw$, where  $y \in O_K$, all the primes occurring  in the divisor of  $y$ are not in $\calW_K$, $w
\in O_{K,\calW_K}$ and all the primes occurring in the divisor of $w$ are in $\calW_K$.  Further, there exists a
positive constant $c$ depending only on $G_0, l_1,\ldots, l_{rn}$, such that for all embeddings $\sigma$ of $K$
into its algebraic closure, $|\sigma(\alpha/\beta)| < |{\mathbf N}_{K/\Q}(y)|^c$, $|{\mathbf N}_{K/\Q}(\beta)| <
|{\mathbf N}_{K/\Q}(y)|^c$ and all the coefficients of the characteristic polynomial of
$\norm_{K(\gamma)/\Q}(\beta)\alpha/\beta$ over $\Q$ with respect to $K(\gamma)$ are also less than
$|{\mathbf N}_{K/\Q}(y)|^c$. (Here, given an element $\beta \in K(\gamma)$, the characteristic
polynomial of $\beta$ is $f(X)=\prod_{j=1}^{rn}(X-\sigma_j(\beta))$, where
$\sigma_1,\ldots,\sigma_{rn}$ are all the embeddings of $K(\gamma)$ into the algebraic closure
of $\Q$.)\qed
\end{lem}%

\begin{notation} We use the following notation in the sequel:%
\begin{itemize}%
\item   Let $m = m_0m_1$ with $m_0$ as defined in Notation \ref{S:notation section}, and $m_1$ as defined in Lemma \ref{le:equiv}.    %
\item Let $Z$ be a positive integer not divisible by any primes of $\calW_K$ and greater than
$ rn \kappa^{nh_K}$, where $\kappa$ is the constant from Corollary \ref{cor:less}.
\item Let $c$ be as in Lemma \ref{le:bounds}.
\end{itemize}
\end{notation}%

\begin{prop}
\label{define}
Consider the following system of equations and conditions where all the variables besides $x_{jm}, x_{km}$, and
$x_{zm}$ take their values in $O_{K,\calW_K}$.
\begin{equation}
\label{eq:mult}
(A, B, C, D, Y, F) \in \Pi
\end{equation}%

\begin{equation}
\label{eq:largeenough}
\exists j, k, z \in \Z_{>0}: \frac{A}{D} = x_{jm};\, \frac{B}{Y} = x_{km};\, \frac{C}{F}=x_{zm}
\end{equation}%

\begin{equation}
\label{eq:classnumber}%
\left (\frac{A}{D}\right )^{h_K} = \frac{A_1}{D_1} ; \ \ \left (\frac{B}{Y}\right
)^{h_K}=\frac{B_1}{Y_1};\ \   \left(\frac{C}{F}\right )^{h_K}=\frac{C_1}{F_1}
\end{equation}%

\begin{equation}%
\label{eq:relprime}
X_1A_1 + U_1D_1=1;\,  X_2B_1 + U_2Y_1=1;\,  X_3C_1 + U_3F_1=1
\end{equation}%

\begin{equation}%
\label{eq:bounds1}%
\frac{v^{h_K}}{Z\prod_{i=0}^{rn}G_i(A_1/D_1-l_i)G_i(x^{2h_K}-l_i)} \in O_{K,\calW_K}
\end{equation}%

\begin{equation}%
\label{eq:bounds2}%
Y_1=(Z^2(v^{5crn})T)^{2h_K}
\end{equation}%

\begin{equation}%
\label{eq:equiv}%
 (F_1B_1-x^{2h_K}Y_1C_1)^{2h_K} =  Y_1^{2h_K+1}w
 \end{equation}%
 We claim that these equations can be satisfied with variables as indicated above only if
$x^{2h_K}$ is an integer.   At the same time, if $x$ is a positive integers the equations above
can be satisfied.
\end{prop}%
\begin{proof}%
Assume that the equations above are satisfied with all the variables except for $x_{jm}, x_{km}, x_{zm}$ taking
values in $O_{K,\calW_K}$.  Then from equation (\ref{eq:mult}) we conclude that $z=jk$. Let
$v^{h_K}=yu$, where $y \in O_K$ and does not have any primes from $\calW_K$ in its divisor and is
not a unit of $O_K$, while all the primes occurring in the divisor of $u$ are from $\calW_K$. We
can assume $y$ is not a unit because from equation (\ref{eq:bounds1}) we know that $v^{h_K}$ is
divisible by  $Z$ which is not a unit of $O_{K,\calW_K}$.
We now combine three inequalities described below. Throughout the proof, we will set $$ N:=|\norm_{K/\Q}(y^c)|. $$
First, from (\ref{eq:bounds1}), by Lemma \ref{le:bounds}, we have that
$$|\norm_{K/\Q}{\dd(A_1/D_1)}| \leq N.$$  Further, from equations (\ref{eq:classnumber}),  we also know that%
$$d_{jm}^{h_K} \, \mid \, \norm_{K/\Q}(\dd(A_1/D_1))$$%
with $d_{jm}$ as in Notation \ref{S:notation section}. Thirdly, from Corollary \ref{cor:less} we find the bound
$$ (jm)^2 \leq \kappa d_{jm}, $$ where $\kappa$ is a fixed positive constant independent of $j$ and $y$, defined in Corollary
\ref{cor:div}. From these three inequalities, we conclude that
\begin{equation}%
\label{star}%
j^{2h_K} < \kappa^{h_k} N^2,
\end{equation}%

We now turn our attention to equation (\ref{eq:equiv}). Write $\nn(Y_1)=(\mathfrak e_0)^{h_K}(\mathfrak e_1)^{h_K}$, where
$\mathfrak e_0$ is an integral divisor not divisible by any prime of $\calW_K$ and $\mathfrak e_1$ is a divisor consisting
of $\calW_K$-primes only. We rewrite (\ref{eq:equiv}) as
\begin{equation}%
\label{eq:numerators}%
 \frac{Y_1}{C_1^{2 h_K}} \cdot w = \left( \frac{F_1 B_1}{Y_1 C_1} - x^{2h_K}\right)^{2 h_K}.%
\end{equation}%
Since by Lemma \ref{le:relprime} and equation (\ref{eq:relprime}) we have that  $Y_1$ and $C_1$ are
coprime in $O_{K,\calW_K}$,  if we consider the non-$\calW_K$ part of
the  numerators of the divisors of the left and right sides of (\ref{eq:numerators}), we see that
 $$\mathfrak e_0 \, \mid \, \displaystyle \nn(\frac{F_1B_1}{Y_1C_1}-x^{2h_K})^2.$$  By Corollary \ref{cor:h_K},
 $$\mathfrak e_0 \, | \, \displaystyle\nn(\frac{F_1B_1}{Y_1C_1}-j^{2h_K})^2$$ and by Lemma
\ref{le:square} we know that $\mathfrak e_0$ is a square of
an integral divisor. Therefore we conclude
\begin{equation}%
\label{eq:divisibility}%
\sqrt{\mathfrak e_0} \, | \, \nn(j^{2h_K}-x^{2h_K}).
\end{equation}%
Next we write $x^{h_K}=\frac{x_1}{x_2}$, where $x_1, x_2 \not = 0$ are relatively prime integers of
$K$. If we clear denominators in (\ref{eq:divisibility}) using $\norm_{K/\Q}(x_2^{2h_K})$ we
get \begin{equation} \label{blub} \sqrt{\mathfrak e_0} \, \mid \, (\norm_{K/\Q}(x_2^{2h_K})j^{2h_K}-\norm_{K/\Q}(x_2^{2h_K})x^{2h_K}).\end{equation}

 We let $H(T)$ be the characteristic polynomial of
$\norm_{K(\gamma)/\Q}(x_2^{2h_K})x^{2h_K}$ over $\Q$ with respect to $K(\gamma)$. Then by (\ref{blub})
\[%
H(\norm_{K(\gamma)/\Q}(x_2^{2h_K})j^{2h_K})^2\equiv 0 \mod \norm_{K(\gamma)/\Q}(\mathfrak e_0)
\]%
 and therefore either
\begin{equation} \label{best}
H(\norm_{K(\gamma)/\Q}(x_2^{2h_K})j^{2h_k})=0
\end{equation}
or
\[%
|H(\norm_{K(\gamma)/\Q}(x_2^{2h_K})j^{2h_K})|^2 \geq |\norm_{K(\gamma)/\Q}(\mathfrak e_0)|.
\]%
However, we can estimate an expression such as $|H(X)|$ by its degree (here, $rn$) times its leading monomial (here, $X^{rn}$) times any bound on its coefficients.
 Now from Lemma (\ref{le:bounds}), we have that the coefficients of the characteristic polynomial of
$\norm_{K(\gamma)/\Q}(x_2^{2h_K})x^{2h_K}$ over $\Q$ with   respect to $K(\gamma)$ are bounded by
$N$.   Therefore, we get
\[%
|H(\norm_{K(\gamma)/\Q}(x_2^{2h_K})j^{2h_K})|^2 \leq |rn N \norm_{K(\gamma)/\Q}(x_2^{2h_K})^{rn} j^{2 h_K rn}|^2.  %|rn\kappa^{rn}\norm_{K(\gamma)/\Q}(y)^{5rnc}|^2 <
%|\norm_{K(\gamma)/\Q}(\mathfrak e_0)|
\]%
But now, we use equation (\ref{eq:bounds1}) and Lemma \ref{le:bounds} again to conclude that
\[%
|\norm_{K/\Q}(x_2^{2 h_K})| < N.
\]%
 From equation (\ref{star}), we find $j^{2 h_K} < \kappa^{h_k}
N^2$, so that if we plug this into the previous inequality, we get
\begin{equation} \label{one} |H(\norm_{K(\gamma)/\Q}(x_2^{2h_K})j^{2h_K})|^2 \leq |rn \kappa^{h_Krn} N^{3rn+1}|^2. \end{equation}

With our definitions, equation (\ref{eq:bounds2}) implies $$\mathfrak e^{h_K}_0 \mathfrak n(\mathfrak e^{h_K}_1) = \mathfrak n (Z^2 y^{5crn} w^{5crn} T)^{2 h_K}.$$
Recall that $Z$ is an integer such that $Z \geq rn \kappa^{nh_K}$. If we now only consider the non-$\calW_K$-part of the equality and take norms and then
$h_K$-th roots, we find
\begin{equation}%
\label{two}%
|\norm_{K(\gamma)/\Q}(\mathfrak e_0)| \geq |\norm_{K(\gamma)/\Q}(r^2n^2 \kappa^{2nh_K}) N^{5rn}|^2 = |(rn \kappa^{2nh_K})^{rn} N^{5rn}|^2. %
\end{equation}%
From (\ref{one}) and (\ref{two}) we conclude that
$$ |H(\norm_{K(\gamma)/\Q}(x_2^{2h_K})j^{2h_K})|^2 < |\norm_{K(\gamma)/\Q}(\mathfrak e_0)|, $$%
In the end, we find that the alternative (\ref{best}) holds, so $H$ has a rational root, and thus all its roots are rational (and equal). Hence
 $x^{2h_K} = j^{2h_K}$ is a rational integer.

In the other direction, suppose that $x = j \in \Z_{>0}$. Set $x_{jm}=\frac{A}{D}, A, D \in O_K$ and
set $x_{jm}^{h_K}=\frac{A_1}{D_1}$, where $A_1, D_1$ are relatively prime elements of $O_K$. Then
the $A_1, D_1$-part of (\ref{eq:relprime}) will be satisfied. Set
\[%
v =Z \prod_{i=0}^{rn}D_1^{r}G_i(A_1/D_1-l_i)G_i(x^{2h_K}-l_i).
\]%
By Lemma  there exists $k \in \Z_{>0}$ such that $Z^4v^{10crn}$ divides $\dd(x_{km})$. Let $z=jk$ and define
$B,C,Y,F,B_1, C_1, Y_1, F_1$ so that (\ref{eq:mult}), (\ref{eq:largeenough}), (\ref{eq:classnumber}) and (\ref{eq:relprime})
are satisfied. Observe that by choice of $k$ we satisfy (\ref{eq:bounds2}) also. Further by Lemma \ref{le:equiv} we have that
$\nn(Y_1)$ divides $\displaystyle \nn(\frac{F_1B_1}{Y_1C_1}-x^{2h_K})^{2h_K}$ and therefore $\nn(Y_1^{2h_K+1})$ divides
$(F_1B_1-x^{2h_K}Y_1C_1)^{2h_K}$. Thus the ratio
\[%
\frac{(F_1B_1-x^{2h_K}Y_1C_1)^{2h_K}}{Y_1^{2h_K+1}} \in O_K \subset O_{K,\calW_K}
\]%
and therefore (\ref{eq:equiv}) is also satisfied.
\end{proof}%

We summarize the previous discussion in the following result.

\begin{thm}\label{blah}
Let $K$ be any number field such that there exists an elliptic curve $E$ defined over $K$ of rank 1 over $K$. Let $W$ be a
Weierstrass equation of $E$ over $K$ and let $t$ be the size of the torsion group of $E(K)$. Let $L$ be any non-trivial
extension of $K$ and let $\calW_K \subset \calP_K$ be any set of primes of $K$ satisfying the
following conditions.%
\be%
\label{thm:ell}%
\item The complement of $\calW_K$ in $\calP_K$ contains all but finitely many elements of the set
$\calV_K(P) =\{\pp_{\ell^j}:\ell \in \calP(\Z), j \in \Z_{>0}\}$ for some point $P \in tE(K)$
of infinite order.%
\item $\calS_{\bad}(P,W,K) \subset \calW_K$.%
\item All but finitely many primes of $\calW_K$ do not have a relative degree one factor in the extension $L/K$.
\ee%
Then $\Z$ can be defined in $O_{K,\calW_K}$ using two universal quantifiers. \end{thm}
\begin{proof}%
This follows from combining the above results with  with Corollary B.10.10 from \cite{Sh34}.
The only point which needs to be made is that we
can existentially define integrality at finitely many primes (see Chapter 4 of \cite{Sh34}) and therefore the
relaxation of assumptions on $\calW_K$ or $P$ will not alter our conclusion.
\end{proof}%
\begin{rem}%
For the construction of
\emph{diophantine} models of $\Z$ in \cite{Poonen} and \cite{PS} to go through, infinitely many elements of the set $\calV_K(P)$
have to be inverted. This is very different from the situation in the above theorem. \end{rem}%

\section{Density computation}
\label{sec:density}

We first compute the density of the set $\calV_K(P)$. For that, we need the following lemma:

\begin{lem}%
\label{le:pindenom}%
Let $l \in \calP(\Q)$ and suppose $\pp \in \calS_{l^{n+1}}\setminus {\calS}_{l^n}$ (if such $\pp$ exists, $n>a_\ell$). Then $l^{n+1} < 3\norm  \pp$.

\end{lem}%
\begin{proof} If $\pp \in {\calS}_{l^{n+1}}\setminus {\calS}_{l^n}$, then  $\pp$ does not divide the discriminant of our
Weierstrass equation and $\tilde{E}$, the reduction of $E \mod \pp$  is
non-singular. Further,
$x_{l^n}$, $y_{l^n}$ are integral at $\pp$, while $\ord_{\pp}x_{l^{n+1}}<0$, $\ord_{\pp}y_{l^{n+1}}<0$. Therefore,
under reduction $\mod \pp$, the image of $l^nP$ is not $\tilde{O}$ -- the image of $O \mod \pp$ , while
$l^{n+1}\tilde{P}=\tilde{O}$. Thus we must conclude that $E(\F_{\pp})$ has an element of order $l^{n+1}$ and
therefore $l^{n+1}| \#E(\F_{\pp})$.  Let $\# \F_{\pp}= \norm{\pp} = q$.
From a theorem of Hasse we know that $\#E(\F_{\pp}) \leq q+1+2 \sqrt{q} \leq 3q$ (see \cite{Silverman}, Chapter V, Section 1, Theorem 1.1).
\end{proof}%

\begin{prop}
The set $\calV_K(P)$ has natural density zero.
\end{prop}
\begin{proof} Recall that $\pp_{\ell^k}$ is a primitive prime divisor of largest norm for $\ell^kP$.
For the proof, we first remark that it is proven in  \cite{Poonen} and \cite{PS} (using properties of Galois representations) that the set of primitive largest norm divisors of $\ell P$
\[%
\calB=\{\pp_{\ell}: \ell \in \calP_{\Q} \land a_{\ell}=1\}
\]%
has a natural density that is zero. To prove the theorem, it therefore suffices to consider the complement of $\calB$ in $\calV_K(P)$, as in the next proposition. It turns out this is much easier:

\begin{lem}%
\label{prop:density}
The natural density of the set  $\calA=\{\pp_{\ell^k}: \ell \in \calP_{\Q}, k \in \Z_{>1} \land k >
a_{\ell}\}$ is zero.
\end{lem}%
\begin{proof}%
For $\pp = \pp_{\ell^k} \in \calA$, the previous Lemma says $3\norm\pp_{\ell^k} > \ell^k$.  Thus,
\[%
\# \{\pp \in  \calA \, : \, \norm{\pp} \leq X \}  \leq   \# \{(\ell,k)  \in \calP_\Q \times \Z_{\geq 2} \, : \,  \ell \leq
\sqrt[k]{3X}\}
\]%
Clearly if $\sqrt[k]{3X} < 2$, there will be no prime $\ell$ with $\ell \leq \sqrt[k]{3X}$.  Thus,
we can limit ourselves to positive integers $k$ such that $k < \log 3X$.

By the Prime Number Theorem (see \cite{L}, Theorem 4, Section 5, Chapter XV), for some positive
constant $C$ we have that $\#\{\ell \in \calP_{\Q}: \ell \leq X\} \leq C {X}/{\log X} $ for all
$X \in \Z_{>0}$.  From the discussion above we now have the following sequence of inequalities:
\begin{eqnarray*}
\{\pp \in  \calA \, : \, \norm{\pp} \leq X \} &\leq& \sum_{k=2}^{\lceil \log 3X \rceil} \#\{\ell \in \calP_{\Q}: \ell \leq \sqrt[k]{3X}\}  \\ &\leq&
\sum_{k=2}^{\lceil \log 3X \rceil} \#\{\ell \in \calP_{\Q}: \ell \leq \sqrt{3X}\} \\ &\leq& \log(3X)[C\frac{\sqrt{3X}}{\log{\sqrt{3X}}}] = \tilde C\sqrt{X}
\end{eqnarray*}
for some positive constant $\tilde C$.  At the same time by  the Prime Number Theorem again we
also know that  for some positive constant $\bar C$ we have
$
\#\{\pp \in \calP_K: \norm \pp \leq X \} \geq \bar C {X/\log X}.
$
Thus the upper density of $\calA$ is
\[%
\limsup_{X \rightarrow \infty}\frac{\#\{ \pp \in A \, : \, \norm \pp \leq X\}}{\#\{\pp \in
\calP_K: \norm \pp \leq X \}} \leq \limsup_{X \rightarrow \infty} \frac{\tilde C\sqrt{X}\log X}{\bar CX} = 0.
\]%.
Hence $\calA$ has a natural density, and it is zero.
\end{proof}
\end{proof}

We can prove Theorem 2 from the introduction:

\begin{thm}
\label{thm:density1}%
Let $K$ be a number field such that there exists an elliptic curve $E$ defined over $K$ of rank 1
over $K$. Then for any $\varepsilon >0$ there exists a set of primes $\calW_K$ of density greater
than $1-\varepsilon$ such that $\Z$ can be defined in $O_{K,\calW_K}$ using two universal
quantifiers.
\end{thm}%
\begin{proof}%
First of all we observe that for any point $P \in E(K)$ of infinite order, the set $\calV_K(P)$ is of natural
density 0 by the previous proposition. Next let $L$ be an extension of $K$ of prime degree $p >
\frac{1}{\varepsilon}$. Then, by the natural version of the Tchebotarev Density Theorem, the set of primes of $K$
having a degree one factor in the extension $L/K$ has natural density $\frac{1}{p}$. Adding primes of $\calV_K(P)$
to this set does not change its density. We apply Theorem \ref{blah} with this $\calW_K$.
\end{proof}%

\section{Proof of the First Main Theorem}

We will now use the following definability results, proofs of which can be found in \cite{Sh36}, \cite{Sh33}
and Section 7.8 of  \cite{Sh34}.
\begin{prop}%
\label{thm:down}
Let $K \not = \Q$ be a number field of one of the following types:%
\be%
\item $K$ is totally real;%
\item $K$ is an extension of degree two of a totally real number field;%
\item There exists an elliptic curve defined over $\Q$ and of positive rank over $\Q$ such that this
curve preserves its rank over $K$;
\ee%
Let $L$ be a totally real cyclic extension of $\Q$ of degree $p$ such that $p$ does not divide $[K_G: \Q]$, where
$K_G$ is the Galois closure of $K$ over $\Q$.  Let $\calW_K$ be a set of $K$ such that all but finitely many
primes in the set do not split in the extension $KL/K$.  Then there exists a set of $K$-primes $\tilde \calW_K$
containing $\calW_K$ and such that $\tilde \calW_K \setminus \calW_K$ has at most finitely many elements, while the
set $O_{K, \tilde \calW_K}\cap \Q$ has a Diophantine definition in $\Q$. \qed
\end{prop}%

We will also need the following property of natural density of sets of primes.
\begin{lem}%
\label{prop:density2}
Let $K$ be any number field, let $\calU_{\Q}$ be a set of rational primes of natural density 0 and
let $\calU_{K}$ be the set of all the primes of $K$ lying above the primes of $\calU_{\Q}$.  Then
the natural density of $\calU_{K}$ is also 0.
\end{lem}%
\begin{proof} This follows from the fact that $\#\{\pp \in \calP_K: \norm{\pp} \leq X\}
= O(X/\log X)$ for any number field.\end{proof}

We are now ready to prove Theorem 1 from the introduction.
First we need a couple of technical propositions  which will allow us to reduce the number of
quantifiers.
\begin{lem}%
\label{allappear}%
Let $M/K$ be a number field extension of degree $n$. Let $\calW_K \subset  \calP_K$, let $\calW_M
\subset \calP_M$ contain all the $M$-factors of primes in $\calW_K$ so that $O_{M,\calW_M}$ is the
integral closure of $O_{K,\calW_K}$ in $K$. Let $\alpha \in O_{M,\calW_M}$ generate $M$ over $K$.
Assume also that the discriminant of the power basis of $\alpha$ is $D$. Then for every $w \in
O_{M,\calW_M}$ we have that either  $$w=\sum_{i=0}^{n-1}a_i\alpha^i, a_i \in O_{K,\calW_K}$$ or
$$Dw=\sum_{i=0}^{n-1}b_i\alpha^i, b_i \in O_{K,\calW_K} \land \exists a_i\mbox{ such that }a_i \not
\equiv 0 \mod D$$ in $O_{K,\calW_K}$. Furthermore, both options cannot hold at the same time and every
element of $O_{K,\calW_K}$ occurs as a coefficient in the first sum.
\end{lem}%
 \begin{proof}%
 By a well-known number-theoretic fact (see for example Lemma B.4.12 of \cite{Sh34}), for any  $w
\in O_{M,\calW_M}$    we have that   $Dw=\sum_{i=0}^{n-1}b_i\alpha^i,
b_i \in O_{K,\calW_K}$.  At the same time, if
\begin{equation}%
\label{coord}
w=\sum_{i=0}^{n-1}a_i\alpha^i, a_i \in O_{K,\calW_K}
\end{equation}%
 and $D$ is not a unit in $O_{K,\calW_K}$, then
the second option cannot hold.  (If $D$ is a unit, then the second option cannot hold in any case.)
Thus,  for each $w$ one of the options holds and both cannot hold at the same time.  Next it is
clear that for any choice $(a_0,\ldots,a_{n-1}) \in O_{K,\calW_K}^n$ we have that
$w=\sum_{i=0}^{n-1}a_i\alpha^i \in O_{M,\calW_M}$ and for each $w$ the choice of the
$n$-tuple $(a_0,\ldots,a_{n-1})$ satisfying (\ref{coord}) is unique.
 \end{proof}%
   \begin{rem}%
Note that the condition $b_i \not \equiv 0 \mod D$ is actually Diophantine, since it is
equivalent to a sentence  $\bigvee_{j=1}^l(b_i \equiv A_j \mod D)$, where the $\{A_j\}$ contains
a representative of every non-zero equivalence class modulo the
principal ideal generated by  $D$ in $O_{K,\calW_K}$.
   \end{rem}%
\begin{prop}  %
\label{reduce}
Let $M, K, \calW_K, \calW_M, \alpha$ be as in Lemma \ref{allappear}.  Assume further that $O_{K,\calW_K}$ is
existentially definable over $O_{M,\calW_M}$.  Let $Z \subset O_{K,\calW_K}$ be definable over
$O_{K,\calW_K}$ by a formula of the form $ \forall \bar X  \exists \bar Y  P( T, \bar X,
\bar Y)$, or  $ \exists \bar U  \forall X  \exists \bar Y  P( T, \bar U, \bar X, \bar
Y)$ where $\forall \bar X$ and $\exists \bar U$, $\exists \bar Y$ represent a sequence of universal
or existential quantifiers respectively.  Assume that the  sequence of universal quantifiers in the
formula is of length less or equal to $n=[K:\Q]$.  Then $Z$ is definable over $O_{M,\calW_M}$ with a
formula using just one universal quantifier.
\end{prop}%
\begin{proof}%
The idea is to encode the variables over which there is universal quantification into a single
universal quantifier over the larger ring, by using them as coefficients in the power basis of
$\alpha$. It is enough to consider the ``translation'' of
\begin{equation}
\label{eq:form}
\exists \bar U \in O_{K,\calW_K}^r  \forall \bar  X \in O_{K,\calW_K}^{\ell}  \exists \bar Y
\in O_{K,\calW_K}^m P( T, \bar U, \bar X, \bar Y)
\end{equation}%
into variables ranging over $O_{M,\calW_M}$. Let $\Gamma(V,\bar Z)$ be a Diophantine definition of $O_{K,\calW_K}$ over
$O_{M,\calW_M}$. Let $\bar U =(u_1,\ldots,u_r), \bar X=(x_0,\ldots,x_{\ell}), \bar R
=(x_{\ell+1},\ldots,x_{n-1}), \ell \leq n-1, \bar Y =(y_1,\ldots,y_m)$. Let
\[%
{\tt F}_1:=  (\Gamma(T,\bar Z_0)=0)
\]%
 \[%
{\tt F}_2:=     (\bigwedge_{i=1}^r\Gamma(u_i, \bar Z_i)=0)
\]    %
   \[%
{\tt F}_3:=            (\bigwedge_{i=0}^{n-1}\Gamma(x_i,\bar Z_{r+i+1})=0)
\]%
     \[%
{\tt F}_4:=       (\bigwedge_{i=1}^m\Gamma(y_i, \bar Z_{r+n+i})=0)
\]%

\begin{equation}
{\tt F}:= \left({\tt F}_1 \land {\tt F}_2 \land {\tt F}_3 \land {\tt F}_4 \right )%
 \end{equation}%
 Let
 \begin{equation}%
  {\tt H} :=  \left  (\bigvee_{i=0}^{n-1}(x_i \not \equiv 0 \mod D ) \right )
 \end{equation}%
  Then (\ref{eq:form}) becomes
\begin{equation}%
\label{translation}%
\exists \bar U \exists \bar Z_0 \exists \bar Z_1,\ldots, \bar Z_r\forall w \exists \bar X \exists
\bar R \exists \bar Z_{r+1},\ldots, \exists \bar Z_{r+n}, \exists \bar Y \exists \bar Z_{r+n+1}
\ldots \bar Z_{r+n+m}
\end{equation}%
\[%
\left ({\tt F}  \land (w =\sum_{i=0}^{n-1}x_i\alpha^i)   \land P(T,\bar U,\bar X,\bar Y)=0 \right )
\]%
\[%
 \lor (  (Dw =\sum_{i=0}^{n-1}x_i\alpha^i)\land  \left (\bigwedge_{i=0}^{n-1}\Gamma(x_i,\bar Z_{r+i+1})=0
\right ) \land {\tt H}  ))
\]%
\end{proof}%
 \begin{thm}%
Let $K$ be a number field of one of the following types:%
\be%
\item $K \not = \Q$ is totally real;%
\item $K$ is an extension of degree two of a totally real number field;%
\item There exists an elliptic curve defined over $\Q$ and of positive rank over $\Q$ such that this
curve preserves its rank over $K$.
\ee%
Let $L$ be a totally real cyclic extension of $\Q$ of degree $p$ such that $p$ does not divide $[K_G: \Q]$, where
$K_G$ is the Galois closure of $K$ over $\Q$.  Let $\calW_K$ be a set of primes of $K$ such that all but finitely many
primes in the set do not split in the extension $KL/K$.  Let $E$ be an elliptic curve defined over $\Q$ and of
rank one over $\Q$. Let $P \in E(\Q)$ be a point of infinite order.  Let $\calV_{\Q}(P)$ be defined as in Theorem
\ref{blah} and let $\calV_K(P)$ be a set of primes of $K$ containing at least one factor for every prime
of $\calV_{\Q}(P)$.  Let $\calU_K = \calW_K \setminus \calV_K(P)$.  Then for some set of $K$-primes $\tilde \calU$
containing $\calU$ and such that $\tilde \calU \setminus \calU$ is a finite set  we have that
\be%
\item $\Z$ is definable  in $O_{K,\tilde \calU_K}$ using one universal quantifier, %
\item for any $\varepsilon >0$, it can be arranged that the natural density of $\tilde \calU_K$ is greater than $1-
\varepsilon$.
\ee%
\end{thm}%
\begin{proof}%
By Theorem \ref{thm:down}, for some set of primes $\tilde \calU$ as described above we have that
$O_{K,\tilde \calU_K} \cap \Q$ is existentially definable in $O_{K, \tilde \calU_K}$. Let
$O_{\Q,\calT_{\Q}}=O_{K,\tilde \calU_K} \cap \Q$. Then given our assumption on $L$, we have that
all but finitely primes of $\calT_{\Q}$ do not split in the extension $L/\Q$. Further, by
construction, $\calV_{\Q}(P) \cap \calT_{\Q}$ is at most finite set. Thus, $\Z$ is definable using
two universal quantifiers in $O_{\Q,\calT_{\Q}}$ and therefore by Proposition \ref{reduce} we can
define $\Z$ in $O_{K,\tilde \calU_K}$ using just one universal quantifier.

Next let $\varepsilon >0$ be given. Then choose $L$ to be of prime degree $p >
\frac{1}{\varepsilon}$ and let $\calW_K$ be the set of {\it all} $K$-primes not splitting
completely in the extension $KL/K$. Then $\calW_K$ will be of natural density $\frac{p-1}{p}$. Next
observe that by Proposition \ref{prop:density2} we have that the density of $\calV_K(P)$ is zero
and therefore removing primes of $\calV_K(P)$ from $\calW_K$ to form $\calU_K$ will not change the
density.

\end{proof}%

\section{Defining Subfields over Number Fields Using One Universal Quantifier}
In this section we will produce another  vertical definability result exploiting properties of
elliptic curves and requiring just one universal quantifier.

\begin{prop}%
Let $M, K, \qq_1, \ldots, \qq_k, \mathfrak Q$ be as in Proposition \ref{extensions}. Let $E$ be an
elliptic curve defined over $K$ such that $\rank\ E(K) >0$.   Then
$K$ is definable over $M$ using just one universal quantifier.
\end{prop}%
\begin{proof}%
Set $r:=[E(M):E(K)]$.
Fix an affine Weierstrass equation $W$ for $E$.  Let $u \in M, \ord_{\qq_i}u >0$ for all
$i=1,\ldots,n$  and consider the following formula:
\[%
\forall z \in M \exists (a_1, b_1), (a_2, b_2)  \in rE(M):
\]%
\[%
\bigwedge_{i=1}^k \ord_{\qq_i}a_1 < \ord_{\qq_i}z\land \ord_{\qq_i}a_2 <z  \land 2\ord_{\qq_i}(u
-\frac{a_1}{a_2}) \geq -\ord_{\qq_i}a_2.
\]%
Here, as above, we identify non-zero points of $E(M)$ with pairs of solutions to the chosen
Weierstrass equation and $rE(M)$ is the set of $r$-multiples of non-zero points of $E(M)$.  Suppose
the formula is true for some value of $u \in M$.  Then by assumption $\frac{x_1}{x_2} \in K$ and  by
Proposition \ref{extensions} we have that $u \in K$.

Now assume that $u \in \Z$,  $u \not =0$ and $u$ is a square.    Let $(x_1,y_1) \in E(M)$ be the
affine coordinates with respect to $W$ of a point $P \in E(M)$  of infinite order.   Then
by Lemma \ref{le:equiv} there exists a positive integer $m_1$ such that for any positive integers
$l, k $,
\[%
\dd(x_{lrm_1}) \Big{|} \nn\left (\frac{x_{lrm_1}}{x_{rklm_1}}-k^2\right)^2
\]%
in the integral divisor semigroup of $M$.  Further, by Corollary \ref{cor:div} and    Lemma
\ref{le:orderchange} we have that for any positive $N$ for some sufficiently large $m$ it is the case
that $\ord_{\qq_i}x_{rmm_1} < -N$ for all $i$.  Finally we note that any positive integer can be
written as a sum of four squares, and any element of $K$ can be expressed as a linear combination of
some basis elements with rational coefficients.
\end{proof}%

We can use the same method of proof over certain subrings of $M$. Then only change
we would have to make is to possibly represent coordinates of points on $E$ as ratios of elements
in the ring. Everything else remains the same since order at a prime is existentially definable in
any $O_{M,\calW_M}$. In this way we arrive at the following.

\begin{prop}
\label{prop:down}
Let $M/K$ be a number field extension.  Assume there exists an elliptic curve $E$
defined over $K$  such that $\rank\ E >0$.  Let $\calW_M$ be any set
of $M$ primes (including the set of all $M$-primes  and the empty set).  Then $O_{M,\calW_M} \cap
K$ is definable over  $O_{M,\calW_M}$ using just one universal quantifier.
\end{prop}%

\begin{rem}%
In connection with the results above we should note that the first-order definability of any
subfield of a number field follows from the  work of Julia Robinson also.  (See
\cite{Robinson:59}.)  However, her definition uses several quantifiers since it proceeds by defining
the algebraic integers over the field first, and then defining $\Z$ over the ring of integers.

Of course for $\calW_M = \emptyset$, $\calW_M$ of finite size and many infinite sets $\calW_M$ we
actually have existential definability. However, we do not have a proof of existential definability
for $\calW_M = \calP_M$. See \cite{Sh33} for more details.
\end{rem}

%%%%%%%%%%%%%%%%%%%%%%%%%%%%%%%%%%%%%%%%%%%%%%%%%%%%%%%%%%%%
% Appendices, if necessary...
%\appendix
%%%%%%%%%%%%%%%%%%%%%%%%%%%%%%%%%%%%%%%%%%%%%%%%%%%%%%%%%%%%

%%%%%%%%%%%%%%%%%%%%%%%%%%%%%%%%%%%%%%%%%%%%%%%%%%%%%%%%%%%%
% bibliography
%%%%%%%%%%%%%%%%%%%%%%%%%%%%%%%%%%%%%%%%%%%%%%%%%%%%%%%%%%%%

\begin{small}
\bibliographystyle{plain}

\end{small}

%%%%%%%%%%%%%%%%%%%%%%%%%%%%%%%%%%%%%%%%%%%%%%%%%%%%%%%%%%%%
% prints affiliation at the end
%%%%%%%%%%%%%%%%%%%%%%%%%%%%%%%%%%%%%%%%%%%%%%%%%%%%%%%%%%%%

\finalinfo

\end{document}